\documentclass[11pt]{amsart}

\usepackage{amsthm, amssymb, amstext, amscd, amsfonts, amsxtra, latexsym, amsmath, comment, tikz,
 mathrsfs, esint, stmaryrd, cancel,pdfcomment}
\usepackage{color}
\usepackage{fancybox}
\usepackage{fullpage}
\usepackage[english]{babel}
\usepackage[latin1]{inputenc}
\usepackage{enumitem}

\numberwithin{equation}{section}
\newtheorem{theorem}{Theorem}
\numberwithin{theorem}{section}
\newtheorem{lemma}[theorem]{Lemma}

\newtheorem{corollary}[theorem]{Corollary}

\theoremstyle{remark}
\newtheorem*{remark}{Remark}
\newtheorem*{remarks}{Remarks}

\newcommand{\Z}{\mathbb{Z}}
\newcommand{\N}{\mathbb{N}}
\newcommand{\Q}{\mathbb{Q}}

\newcommand{\R}{\mathbb{R}}
\newcommand{\ch}{\operatorname{ch}}
\newcommand{\re}{\operatorname{Re}}
\newcommand{\im}{\operatorname{Im}}
\newcommand{\sgn}{\operatorname{sgn}}
\newcommand{\Arg}{\operatorname{Arg}}

\DeclareMathOperator{\erf}{erf}
\DeclareMathOperator{\imm}{Im}
\DeclareMathOperator{\ree}{Re}

\renewenvironment{proof}[1][Proof]{\begin{trivlist} \item[\hskip \labelsep {\bfseries #1:}]}{\qed\end{trivlist}}

\title{Asymptotic behavior of partial and false theta functions arising from Jacobi forms and regularized characters}

\author{Kathrin Bringmann}
\address{Mathematical Institute\\University of
Cologne\\ Weyertal 86-90 \\ 50931 Cologne \\Germany}
\email{kbringma@math.uni-koeln.de}
\author{Amanda Folsom}
\address{Department of Mathematics and Statistics \\  Amherst College \\  Amherst, MA 01002 \\ U.S.A.}
\email{afolsom@amherst.edu}
\author{Antun Milas}
\address{Department of Mathematics and Statistics \\
SUNY-Albany  \\
Albany, NY 12222 \\ U.S.A.}
\email{amilas@albany.edu}
\thanks{ The research of the first author was supported by the Alfried Krupp Prize for Young University Teachers of the Krupp foundation and the research leading to these results has received funding from the European Research Council under the European Union's Seventh Framework Programme (FP/2007-2013) / ERC Grant agreement n. 335220 - AQSER.  The second author is grateful for the support of  NSF CAREER grant DMS-1449679, and for the hospitality provided by the Max Planck Insitutute for Mathematics, Bonn, and the Institute for Advanced Study, Princeton, under NSF grant DMS-1128155.  The third author was partially supported by the Simons Foundation Collaboration Grant for Mathematicians ($\#$ 317908).}

\begin{document}

\maketitle
\begin{abstract}  
We prove several asymptotic results for partial and false theta functions arising from Jacobi forms, as the modular variable $\tau$ tends to $0$ along the imaginary axis, and the elliptic variable $z$ is unrestricted in the complex plane. 
We observe that these functions exhibit Stokes' phenomenon -  the  asymptotic behavior of these  functions sharply differs depending on where the elliptic variable $z$ is located within the complex plane.
We apply our results to study  the asymptotic expansions of regularized characters  and quantum dimensions of the $(1,p)$-singlet vertex operator algebra coming from conformal field theory. This, in particular, recovers and extends several results from \cite{CMW} pertaining to regularized quantum dimensions, which served as a main source of motivation.  
\end{abstract}
\section{Introduction}  
The main objects of interest in this paper are the (Jacobi) \emph{partial theta functions} 
\begin{align}\label{def_Fdell}
F_{d,\ell}\left(z;\tau\right):=\sum_{n \geq 0} \zeta^{\ell n+d}q^{(\ell n+d)^2},
\end{align} 
defined for $d\in\Q$ and $\ell\in\N$.  {Here and throughout, we let} $\zeta:=e^{2 \pi i z}$ and $q:=e^{2\pi i \tau},$ with $z\in \mathbb C$ and $\tau \in \mathbb H$, the complex upper half-plane.  These functions are aptly {named: we find weight $1/2$ modular Jacobi theta functions (up to suitable changes of variables) if the sum in (\ref{def_Fdell}) is extended to be over the full lattice of integers.} Despite their non-modularity, functions like $F_{d,\ell}$ and their specializations have a rich history.  In particular, they are known to  play fundamental roles within the theory of $q$-hypergeometric series and integer partitions in number theory, dating back to the time of Rogers and Ramanujan, and continuing into the present day \cite{Alladi, Andrews, BerndtYee, Fine}.  More recently, we have begun to understand partial theta functions within the theory of modular forms, as they are connected to mock modular forms and also quantum modular forms \cite{ FOR,M, ZagierQ}. Specializations of the partial theta functions $F_{d,\ell}$ are also intimately related to Eichler integrals of modular forms \cite{BringmannRolen, ZagierVas, ZagierQ}.  Outside of number theory, partial theta functions appear in connection to topological invariants of $3$-manifolds \cite{Hikami, LawrenceZagier, ZagierVas}, as generating functions for colored Jones polynomials for alternating knots \cite{GarLe}, and in the representation theory of vertex algebras, the last of which we elaborate upon below.  

In all of the above aspects, it is important to understand the asymptotic properties of partial theta functions.  For example, in his second notebook \cite[p. 324]{BerndtRam},    Ramanujan claimed an asymptotic expansion  for the partial theta function  
\begin{equation*}
2\sum_{n \geq 0} (-1)^n q^{n^2+n} =  1+T+T^2+2T^3+5T^4+\cdots
\end{equation*} with $q=:\frac{1-T}{1+T}$, as $T\to 0^+$,
  combinatorial properties   and generalizations of which  have been studied in \cite{BerndtKim, Galway, Stanley}.  
	
	Similar expansions have been important within the theory of quantum modular forms, for example, {due to work of Zagier \cite{ZagierVas},} the Eichler integrals of weight $k$ cusp forms $g(\tau) = \sum_{n\geq1} a_g(n) q^n$,  close relatives to partial theta functions,  are known to satisfy $(N\in\mathbb{N}_0)$ (see also  \cite{BringmannRolen}) 
		\begin{align}\label{eqn_eichlerasy} \sum_{n \geq 1} a_g(n)n^{1-k} q^n   = \sum_{n=0}^N \frac{(-x)^n}{n!} L_g\left(e^{\frac{2\pi i \ell}{m}};k-1-n\right) +O\left(x^{N+1}\right)\end{align}
with $\tau = \frac{\ell}{m} + \frac{ix}{2\pi}$, $\frac{\ell}{m} \in \mathbb Q$, as $x\to 0^+$,  where the asymptotic coefficients in (\ref{eqn_eichlerasy}) are given in terms of twisted $L$-values. The asymptotic properties of the partial theta functions 
$$\phi_m^{(a)}(\tau) := m \sum_{n \geq 0} \chi_{2m}^{(a)}(n) q^{\frac{n^2}{4m}},$$   where $\chi_{2m}^{(a)}$ are certain characters, were similarly studied in \cite{CMW, Hikami, LawrenceZagier, ZagierVas},  in connection to topological invariants of $3$-manifolds, representation theory, and quantum modular forms.

The typical situation involves partial theta functions obtained by specializing the variable $z$  in (\ref{def_Fdell}) to be a  point in $\mathbb Q \tau + \mathbb Q$.  This produces a one-variable function in $\tau$, and  the asymptotics of these resulting functions are studied as  $\tau\to 0$.    Our first set of results, which are both of independent interest and of interest in representation theory as we discuss below, gives asymptotic expansions for the two-variable Jacobi partial theta functions $F_{d,\ell}(z;it)$ for any $z\in \mathbb C$ as $t \to 0^+$. These functions exhibit  Stokes' phenomenon, in that their asymptotic properties sharply differ depending on where in the complex plane $z$ lies.   To describe this, we write $z=(z_0+j)/\ell,$ where $j\in\mathbb Z$, $z_0=x_0+iy_0,$ with $x_0, y_0 \in \mathbb R$ and $-1/2<x_0\leq 1/2$.   Our results are also phrased using the differential operator $ \mathcal D_z :=\frac{1}{2\pi i} \frac{\partial}{\partial z}$, and the Bernoulli polynomials $B_n(x)$. We note that some special cases of asymptotic 
expansions of $F_{d, \ell}$ have already been investigated in the literature. For instance, the authors in \cite{BK} obtained its asymptotic expansion formula in the $|z|<\frac{1}{4\ell}$ region for $d\in \mathbb Q^+$.

\begin{theorem}\label{Fasex}  
We have the following behavior  as $t \to 0^+$ for any $N\in \mathbb N_0$.
\begin{itemize}[leftmargin=*,align=left]
\item[\rm{(i)}] 
If  $\imm (z)>0  \text{ or }\left(\imm (z)<0 \text{ and } |x_0|>|y_0|\right),    \text{ or } z\in \mathbb{R}\setminus\frac{1}{\ell}\mathbb{Z},$  then  
\begin{align*}
F_{d,\ell}\left(z;it\right)&=\sum_{a=0}^N \mathcal{D}_z^{2a}\left(\frac{\zeta^d}{1-\zeta^\ell}\right)\frac{(-2\pi t)^a}{a!} + O\left(t^{N+1}\right).\end{align*}
\item[\rm{(ii)}] If    $\imm(z)<0$ and $|x_0|\le|y_0|$, then
  \begin{align*}
F_{d,\ell}\left(z;it\right) &= \left(2\ell^2t\right)^{-\frac12}e^{\frac{2\pi ijd}{\ell}-\frac{\pi z_0^2}{2\ell^2t}}\sum_{\substack{n\in\mathbb{Z}\\|n+x_0|\le|y_0|}}e^{-\frac{\pi n^2}{2\ell^2t}-\frac{\pi z_0n}{\ell^2t}-\frac{2\pi ind}{\ell}}  
\\ & \hspace{.5in} +  \sum_{a=0}^N \mathcal{D}_z^{2a}\left(\frac{\zeta^d}{1-\zeta^\ell}\right)\frac{(-2\pi t)^a}{a!} + O\left(t^{N+1}\right).
\end{align*}

\item[\rm{(iii)}]
If $z\in \frac{1}{\ell}\mathbb Z$, then 
\[
F_{d,\ell}\left(z; it\right)=\frac{1}{2}\left(2\ell^2 t\right)^{-\frac{1}{2}}\zeta^d-\zeta^d\sum_{a=0}^{N} \frac{B_{2a+1}\left(\frac{d}{\ell}\right)}{2a+1}\frac{\left(-2\pi\ell^2 t\right)^a}{a!} + O\left(t^{N+1}\right).
\]

\end{itemize}

\end{theorem}
As a corollary,
we deduce the asymptotic behavior of the Jacobi partial theta functions $F_{d,\ell}$, which are again dependent on the location of $z\in\mathbb C$. 
\begin{corollary}\label{5.2}
We have the following behavior, as $t\to 0^+$.
\begin{itemize}[leftmargin=*,align=left]
\item[\rm{(i)}] If $\imm(z)>0$ or $(\imm(z)<0$ and $|x_0|>|y_0|)$ or $(z\in\mathbb{R}\setminus \frac{1}{\ell}\Z)$,  then
$$
F_{d,\ell}\left(z;it\right)\sim\frac{\zeta^d}{1-\zeta^\ell}.
$$
\item[\rm{(ii)}] If $\imm(z)<0,~ |x_0|\leq |y_0|,$ and $~x_0\ne1/2$, then
 $$
F_{d,\ell}\left(z;it\right)\sim\left(2\ell^2t\right)^{-\frac12}e^{\frac{2\pi ijd}{\ell}-\frac{\pi z_0^2}{2\ell^2t}}.
$$

\item[\rm{(iii)}] If $\imm(z)<0,~ |x_0|\leq |y_0|,$ and $~x_0=1/2$, then
$$
F_{d,\ell}\left(z;it\right)\sim2\left(2\ell^2t\right)^{-\frac12}\cos\left(\frac{\pi}{\ell}\left(d+\frac{y_0}{2\ell t}\right)\right)e^{-\frac{\pi}{8\ell^2t}+\frac{\pi y_0^2}{2\ell^2t}+\frac{\pi i (2j+1)d}{\ell}}.
$$

\item[\rm{(iv)}] If $z\in \frac{1}{\ell}\mathbb Z$, then 
$$
F_{d,\ell}\left(z;it\right)\sim\frac{\zeta^d}{2\ell(2t)^{\frac{1}{2}}}.
$$
\end{itemize}
\end{corollary}
 
  \emph{False theta functions}, similar to partial theta functions, are usually defined as sums over a full lattice, but sign changes are made in the summands so that modularity properties are lost. These functions also enjoy a rich history within the theory of $q$-series and integer partitions, and are often intertwined with partial theta functions \cite{AndrewsBook}.  A  {typical} example of a false theta function is
  $$ \sum_{ n \in \mathbb{Z} } {\rm sgn}(n) q^{( \ell n+d)^2},$$
where ${\rm sgn}(x):=1$ for $x \geq 0,$ and ${\rm{sgn}}(x):=-1$ for $x <0$, which may also be viewed as a difference of two partial theta functions.
 In this paper, we are particularly interested in the following difference of two Jacobi partial theta functions    
 $$ G_{d,\ell}(z;\tau):=F_{d,\ell}(z;\tau)-F_{-d,\ell}(z;\tau).$$
Observe that $G_{d, \ell}(0;\tau)+q^{d^2}$ is  a  false theta function as above.
 In Section   \ref{sec_fgsummary} (Theorem \ref{Gasex} and Corollary \ref{5.4}), we establish asymptotic expansions and asymptotic behavior of the functions $G_{d,\ell}$  analogous to Theorem \ref{Fasex} and Corollary \ref{5.2}.   
 All of these results are of independent interest, and also have consequences in representation theory, which we now discuss. 

Partial and false theta functions have recently appeared in the representation theory of vertex algebras in the study of {\em regularized} characters of 
$(1,p)$-singlet vertex operator algebra  \cite{BM,CM,CMW} $(p\geq 2)$.
Previously, in \cite{M},  it was observed that the usual characters ${\rm ch}[M_{r,s}](\tau)$ of atypical modules $M_{r,s}$ ($M_{1,1}$ the singlet vertex algebra),  can be written as 
quotients of differences of two partial theta series and the Dedekind $\eta$-function.
These characters are interesting from several different standpoints. 
For example, they admit elegant representations as multi-hypergometric $q$-series and are also quantum modular forms \cite{ZagierQ}
with quantum set $\mathbb{Q}$ (for details see \cite{BM}).
In \cite{CM}, motivated by developments surrounding the Verlinde formula in conformal field theory, atypical and typical characters are regularized by using a new complex parameter $\varepsilon$, which can be also viewed as the $U(1)$-charge in physics. The resulting expression, denoted by $\ch[M_{r,s}^\varepsilon](\tau)$ ($r \in \mathbb{Z}$, $1 \leq s \leq p-1$), has in the numerator  a difference of two partial theta functions discussed earlier. This numerator, denoted in Section 6 by $\mathcal{C}_{r,s}(\varepsilon; \tau)$, gives rise to a false theta function. Indeed, if  specialized at $\varepsilon=0$, $\mathcal{C}_{1,s}$ is a false theta function as above, while for $r \neq 1$, $\mathcal{C}_{r,s}$ can be written as the sum 
of a false theta function and a finite $q$-series \cite{BM}.
Interestingly, the regularized characters admit a certain modular-type transformation formula if $\varepsilon \notin i \mathbb{R}$  \cite{CM}. 
This fact was instrumental for proving the Verlinde formula of characters  \cite{CM,CMW}, which is conjecturally isomorphic to the Grothendieck ring of the category  of modules for the singlet vertex algebra.
It is known that ordinary fusion ring admits one-dimensional representation coming from quantum dimensions, thus in \cite{CM}, 
{ the} regularized quantum dimension of $M$ was defined as
\begin{equation} \label{quantum-dim} 
{\rm qdim}(M^\varepsilon):=\lim_{t \to 0^+} \frac{\ch[M^\varepsilon](i t)}{\ch[V^\varepsilon](i t)},
\end{equation}
now depending on $\varepsilon$. Without going into details, we only mention that regularized quantum dimensions define a representation of the Verlinde algebra of characters. Very recently, in \cite{CMW}, regularized quantum dimensions of irreducible modules for the singlet algebras were computed 
on a certain subset of the $\varepsilon$-plane. It was observed that quantum dimension have peculiar  { properties} in different regions on the $\varepsilon$-plane, roughly 
corresponding to ${\rm Re}(\varepsilon)>0$ and ${\rm Re}(\varepsilon)<0$.

In this paper, we extend and generalize several results from \cite{CMW}. Firstly, in Theorem \ref{6.2}, we determine explicit (full) asymptotic expansions of all regularized 
irreducible characters simply as  corollaries to the more general asymptotic formulas for the Jacobi partial and false theta functions (Theorem \ref{Fasex}, Corollary \ref{5.2}, Theorem \ref{Gasex}, and Corollary \ref{5.4}). These results immediately imply several  properties observed earlier in \cite{CMW}, including Stokes' phenomenon.  Moreover, we extend  known formulas for 
regularized quantum dimensions in \cite{CMW} to the whole $\varepsilon$-plane including the imaginary axis. To describe this, we  write $\varepsilon=(\varepsilon_0+ik)/\sqrt{2p} $, with $k \in \mathbb{Z}$, $\varepsilon_0 = u_0 + i v_0$ with $u_0, v_0 \in \mathbb R,$ and $-1/2 < v_0 \leq 1/2 $. \\
\indent
Our next result gives formulas for regularized quantum dimensions - as defined in (\ref{quantum-dim}) - of the $(1,p)$-singlet algebra modules, both typical and atypical. We point out that explicit formulas for ${\rm ch}[F^\varepsilon_{\lambda}](\tau)$ and ${\rm ch}[M_{r,s}^\varepsilon](\tau)$  are given in (\ref{Fl}) and (\ref{atypical}), respectively. As in Theorem \ref{Fasex} and Corollary \ref{5.2}, Theorem \ref{qdim} exhibits Stokes' phenomenon. 
\begin{theorem} \label{qdim} Assume the notation and hypotheses above.
\begin{itemize}[leftmargin=*,align=left] 
\item[\rm{(i)}]    If one of the following are true:  
$\ree(\varepsilon )<0$ {or}  $(\ree(\varepsilon)>0$ {and} $|v_0| > |u_0|)$  {or} $(\ree (\varepsilon) >0, p|k, |u_0|<1-|v_0|$,  {and} $v_0 \not = 1/2)$ {or } $(u_0 =0$ {and} $v_0 \not=0)$,
 {then}   
\begin{align}\label{qdimM13i}
{\rm qdim}[M^\varepsilon_{r,s}] &= e^{\pi \varepsilon \sqrt{2p}(1-r)} \frac{\sinh\left(\frac{\sqrt{2}\pi s\varepsilon}{\sqrt{p}}\right)}{ \sinh\left(\frac{\sqrt{2}\pi \varepsilon}{\sqrt{p}}\right) }, \\ \label{qdimF13i} {\rm qdim}[F^\varepsilon_\lambda] &=  e^{2\pi \varepsilon \left(\lambda-  \sqrt{\frac{p}{2}}+\sqrt{\frac{1}{2p}}  \right)} \frac{\sinh\left(\sqrt{2p}\pi \varepsilon\right)}{ \sinh\left(\frac{\sqrt{2}\pi \varepsilon}{\sqrt{p}}\right) }.\end{align}

\item[\rm{(ii)}]   If $\ree(\varepsilon)>0$, $p\nmid k$,  $|v_0| \leq |u_0|,$ and $v_0\ne1/2$,  then

$${\rm qdim}[M^\varepsilon_{r,s}]= (-1)^{k(r+1)}\frac{\sin\left(\frac{\pi ks}{p}\right)}{  \sin\left(\frac{\pi k}{p}\right)},  \ \ \ \   {\rm qdim}[F^\varepsilon_\lambda]=0.$$
\item[\rm{(iii)}]
 {If $\ree(\varepsilon)>0,~p| k,~1-|u_0| \leq |v_0|\leq |u_0|,\text{ and } v_0\ne1/2$, then }  

$${\rm qdim}[M^\varepsilon_{r,s}]=(-1)^{(r+1)(k+1)+\frac{k(s+1)}{p}}\frac{\sin\left(\frac{\pi s}{p}\right)}{\sin\left(\frac{\pi}{p}\right)}, \ \  \ {\rm qdim}[F^\varepsilon_{\lambda}]=0.$$

\item[\rm{(iv)}]  {If $\ree(\varepsilon)  >0$, $v_0 =1/2$, and $|u_0| \geq 1/2$, then  $${\rm qdim}[F^\varepsilon_\lambda]=0,$$ and ${\rm qdim}[M^\varepsilon_{r,s}]$ exists 
if and only if }
\[
 \tan\left(\frac{\pi s }{2 p}+\frac{\pi r}{2}\right) \tan\left(\frac{\pi}{2p}\right)
 = \tan\left(\frac{\pi(2k+1)s}{2p}+\frac{\pi r}{2}\right)  \tan\left(\frac{\pi(2k+1)}{2p} \right).
\]

 {If this condition is satisfied, then}
$${\rm qdim}[M_{r,s}]= \frac{(-1)^{(r+1)k+1}  \sin\left( \frac{\pi s}{2 p}+\frac{\pi r}{2} \right)\cos\left( \frac{\pi}{2p} (2k+1)s+\frac{\pi r}{2}\right)}{   \cos\left(\frac{\pi}{2p} \right)\sin\left( \frac{\pi}{2p}(2k+1) \right)}. $$
\item[\rm{(v)}]  {If $\varepsilon_0=0$ and $p \nmid ks$, then}  
\[
{\rm qdim}[M^\varepsilon_{r,s}]=  (-1)^{k(r+1)} \frac{\sin \left(\frac{\pi ks}{p} \right)}{\sin \left(\frac{\pi k}{p} \right)}, \ \ \ {\rm qdim}[F^\varepsilon_\lambda]=0.
\]

\item[\rm{(vi)}]  {If   $ \varepsilon_0 =0$, $p \mid ks$, then}  
\begin{enumerate}
\item[\rm{(a)}]
 {if $k \notin p \mathbb{Z}$, then}
\[
{\rm qdim}[M^\varepsilon_{r,s}] =0, \ \  {\rm qdim}[F^\varepsilon_\lambda]=0;
\]
\item[\rm{(b)}] {if 
$k \in p (2 \mathbb{Z})$, then }
\[
{\rm qdim}[M^\varepsilon_{r,s}] =s, \ \  {\rm qdim}[F^\varepsilon_\lambda]=e^{\frac{2 \pi i k \lambda}{\sqrt{2p}}} p;
\]
\item[\rm{(c)}] { if
$k \in p(2\mathbb{Z}+1)$, then}
\[
{\rm qdim}[M^\varepsilon_{r,s}] =(-1)^{(s+1)+p(r+1)}s, \ \   {\rm qdim}[F^\varepsilon_\lambda]=  e^{\frac{2 \pi i k \lambda}{\sqrt{2p}}}p .
\]
\end{enumerate}
\end{itemize}

\end{theorem}
\begin{remarks}\noindent 
\begin{itemize}[leftmargin=*,align=left]
\item[(i)] Parts (i) (not including the $u_0=0, v_0 \neq 0$ region), (ii), and (iii) of Theorem \ref{qdim}  were  proven in \cite{CMW} by using a different method. \smallskip
\item[(ii)] Part (vi) of Theorem \ref{qdim} for $k=0$ was given in \cite{CM}. \smallskip
\item[(iii)]   In part (iv), there are several choices of parameters, some of which are dependent on $k$, which yield a 
solution. For instance, if $r$ is odd and $s=1$, then the quantum dimension of $M_{r,s}$ is $1$, and if $r$  is even and $s=p-1$, then the quantum dimension is $-1$.
\end{itemize}
\end{remarks}

This paper is structured as follows. In Section 2, we introduce basic functions and prove several technical lemmas. In Sections 3 and 4, we determine
asymptotic properties  of the partial theta functions $F_{d,\ell}$ and $G_{d,\ell}$ for ${\rm Im}(z) \neq 0$ and   $\imm(z) = 0$, respectively. In particular, their  asymptotic  behavior is stated in Corollaries \ref{3.4}, \ref{3.5}, and \ref{cor_fgreal}.  In Section \ref{sec_fgsummary}, we prove Theorem \ref{Fasex} and Corollary \ref{5.2}, and establish the analogous results for the functions $G_{d,\ell}$ in Theorem \ref{Gasex} and Corollary \ref{5.4}.   Finally, in Section \ref{sec_singlet}, we prove Theorem \ref{qdim}, and establish related asymptotic expansions in Theorem \ref{6.2} and Corollary \ref{6.3}.
 
\section{Auxiliary functions} \label{sec_aux}In this section, we provide some preliminary results on certain functions required to prove our main results.  Namely, we establish the asymptotic properties of some functions defined using the error function,  we give some transformation properties for various theta functions, and we establish expansions for some special functions.
\subsection{Properties of error functions} The error function $\operatorname{erf}$ is defined,  for $w\in \mathbb C$, by 
$$\operatorname{erf}(w) := \frac{2}{\sqrt{\pi}} \int_0^w e^{-t^2}dt.$$ 

We establish asymptotic properties of certain functions defined using $\operatorname{erf}$ below;  in doing so, we make use of the following expansions { \cite[7.6.2, 7.12 (i)]{NIST} } of $\operatorname{erf}$ and the complementary error function $\operatorname{erfc}(w) := 1-\operatorname{erf}(w)$,
the first of which converges for any $w\in \mathbb C$, and the second of which holds for any $N \in \mathbb N_0$,
 as $w\to 0$, for $\lvert \mathrm{Arg}(w) \rvert <3\pi /4$,
 \begin{align}\label{NIST2}
\operatorname{erf}(w) & =  e^{-w^2} \sum_{n\geq 1} \frac{w^{2n-1}}{\Gamma \left(n+\frac{1}{2}\right)}, \\
 \label{NIST}
 \text{erfc}(w) & = \frac{e^{-w^2}}{\sqrt{\pi}} \sum_{m=0}^N \frac{(-1)^m \left( \frac{1}{2}\right)_m}{w^{2m+1}} + O\left(w^{-2N-3}\right).
 \end{align}  
\begin{lemma}\label{2.1} 
For $w \in \mathbb{C}$, we have that
\begin{equation}
\label{erid}
\sum_{n\geq 0} \frac{(2w)^{n} \Gamma \left( \frac{n+1}{2} \right)}{n!} = \sqrt{\pi} e^{w^2} \left(1+\operatorname{erf}\left(w \right) \right).
\end{equation}
\end{lemma}

\begin{proof}
The statement follows by splitting the sum in (\ref{erid}) into even and odd terms, using the Taylor expansion for $e^{w^2}$, as well as  (\ref{NIST2}).
\end{proof}

\indent
Next, we define the function
 \begin{equation*}
 F(t,w) := \frac{w}{\sqrt{t}} e^{\frac{w^2}{t}} \left( 1+\text{erf}\left(\frac{w}{\sqrt{t}}\right) \right), 
 \end{equation*} and determine its asymptotic behavior in Lemma \ref{2.2}.

\begin{lemma}\label{2.2} We have the following asymptotic behavior,  as $t\to 0^+$.
\begin{enumerate}[leftmargin=*,align=left]
 \item[\rm{(i)}] If $\lvert \mathrm{Arg}(w) \rvert \leq \pi /4$, then 
 \[F(t, w) -\frac{2w}{\sqrt{t}} e^{\frac{w^2}{t}} =-\frac{1}{\sqrt{\pi}} + O(t).\]
 \item[\rm{(ii)}]  If $\lvert \mathrm{Arg}(w) \rvert > \pi /4$, then \[F(t,w)= - \frac{1}{\sqrt{\pi}} \left( 1-\frac{t}{2w^2} \right) + O\left(t^2\right). \]
\end{enumerate}
\end{lemma}

\begin{proof} 
\ \\ (i) If $\lvert \mathrm{Arg}(w) \rvert \leq \pi /4$, 
then the claim follows directly from \eqref{NIST}. 
 
\noindent (ii) If $\lvert \mathrm{Arg}(w) \rvert > \pi /4$, then $\lvert \mathrm{Arg}(-w) \rvert < 3\pi/4$ and the claim follows similarly. 
\end{proof}

\noindent {\it Remark.}
The change in the behavior of the asymptotic expansion of $F(t,w)$ across the boundary
$|\Arg(w)|=\pi/4$ is an example of Stokes' phenomenon. {The lines $\Arg(w)=\pm \pi/4$ are called {\em anti-Stokes} lines}.
 
\subsection{Jacobi and partial theta functions} In this section, we  provide a   transformation property for  the Jacobi theta function  
\begin{equation}\label{Theta}
\Theta (z;\tau) := \sum_{n\in \Z} (-1)^n \zeta^n q^{n^2} ,
\end{equation}  and also  establish a shifting property of  the partial theta functions $F_{d,\ell}$.\\
\indent
The function $\Theta$ satisfies the following well-known modular transformation property (see \cite{Mu}, Chapter 1), which we make use of:  
\begin{equation}\label{thetatrans}
\Theta (z;\tau) = (-2i \tau)^{-\frac{1}{2}} \sum_{n  \in 1 +2\mathbb{Z}} e^{-\frac{\pi i}{8 \tau} (n+2z)^2} .
\end{equation}
\indent
 A direct calculation yields the following shifting property for the functions $F_{d,\ell}$.  
\begin{lemma}\label{lem_fshift} For $m\in\mathbb Z,$ we have that
$$
F_{d+m\ell, \ell}(z;\tau)=F_{d, \ell}(z;\tau)-\sum_{a\geq 0}\mathcal{D}_z^{2a}\left(\frac{\zeta^{d}\left(1-\zeta^{\ell m}\right)}{1-\zeta^{\ell}}\right) \frac{(2\pi i\tau)^a}{a!}.
$$
\end{lemma}
  
In Section 3, we require  the  expansions of certain derivative functions similar to those appearing above.  Using the Bernoulli polynomial generating function (see 24.2.3 in \cite{NIST}), we establish the following lemma. 
\begin{lemma}\label{derivative}
For $z\in\mathbb{C}\setminus \{0\},$ and $a\in\mathbb{N}_0$, we have that
\begin{align}\label{eqn_dber}
\mathcal{D}_z^{2a}\left(\frac{\zeta^d}{1-\zeta^\ell}\right)=-\ell^{2a}\left(\frac{(2a)!}{\left(2\pi i\ell z\right)^{2a+1}}+\sum_{b\geq 0}\frac{\left(2\pi i\ell z\right)^{b}B_{2a+b+1}\left(\frac{d}{\ell}\right)}{b!(2a+b+1)}\right).
\end{align}
Moreover, for $a\in \mathbb N_0$, we have that
\begin{align}\label{eqn_dber0}\lim_{z\to 0}\left(\mathcal{D}_z^{2a}\left(\frac{\zeta^d}{1-\zeta^\ell}\right)+\frac{(2a)!}{\ell \left(2\pi i z\right)^{2a+1}}\right)= -\ell^{2a}\frac{B_{2a+1}\left(\frac{d}{\ell}\right)}{ 2a+1}.\end{align}
\end{lemma} 

\section{ Asymptotic expansions of $F_{d,\ell}$ and $G_{d,\ell}$ if  $\im(z)\neq 0$ } \label{sec_sec3} In this section, we establish the asymptotic expansions of the partial theta functions  $F_{d,\ell},$ and of  the functions  $G_{d,\ell},$ if $\im(z) \neq 0$, making use of some of the results established in Section \ref{sec_aux}.   In Section \ref{sec_31}, we consider  the case $\im(z) >0$, and in Section \ref{sec_32}, we treat the case $\im(z)<0$. Recall that we  write $z=(z_0+j)/\ell,$ where $j\in\mathbb Z$, $z_0=x_0+iy_0,$ with $x_0, y_0 \in \mathbb R$ and $-1/2<x_0\leq 1/2$.
\subsection{$\im(z)> 0$}\label{sec_31}  In Lemma \ref{3.1} below, we establish the asymptotic expansions of the functions $F_{d,\ell}$ and  $G_{d,\ell}$ in the case $\im(z)>0$.
\begin{lemma}\label{3.1}
We have, for $\tau \in \mathbb H$  and $z\in\mathbb{C}$ with $\im(z)> 0$, 
\begin{align*}
F_{d,\ell}(z;\tau)&=\sum_{a\geq0} \mathcal{D}_z^{2a}\left(\frac{\zeta^d}{1-\zeta^\ell}\right)\frac{\left(2\pi i\tau\right)^a}{a!},\\
G_{d,\ell}(z;\tau)&=2i\sum_{a\geq 0}\mathcal{D}_z^{2a}\left(\frac{\sin\left(2\pi dz\right)}{1-\zeta^\ell}\right)\frac{\left(2\pi i\tau\right)^a}{a!}.
\end{align*}
\end{lemma}

\begin{proof} 
Using the definition of the functions $F_{d,\ell}$, we have
 \begin{align}\label{eqn_fexpzplus}
 F_{d,\ell} (z;\tau) =\sum_{a\geq 0} \frac{(2\pi i\tau)^a}{a!}  \sum_{n\geq 0} (\ell n+d)^{2a} \zeta^{\ell n+d}  = 
 \sum_{a\geq 0}  \mathcal{D}_z^{2a} \left(\frac{\zeta^d}{1-\zeta^\ell} \right)\frac{(2\pi i\tau)^a}{a!}.
  \end{align}
This establishes the claimed result for the functions $F_{d,\ell}$.  The statement  for $G_{d,\ell}$ then follows from (\ref{eqn_fexpzplus}) and the definition of the functions $G_{d,\ell}$.
\end{proof}

\indent
From Lemma \ref{3.1} we deduce the following asymptotic behavior of the functions $F_{d,\ell}$ and  $G_{d,\ell}$.
\begin{corollary}\label{cor_fgasypos}
For $\im(z)> 0$, as $t\to 0^+, $ we have that
\begin{equation*}
F_{d,\ell}(z;it)\sim \frac{\zeta^d}{1-\zeta^\ell},\qquad\qquad
G_{d,\ell}(z;it)\sim \frac{2i\sin\left(2\pi dz\right)}{1-\zeta^\ell}.
\end{equation*}
\end{corollary}

\subsection{$\im(z)< 0$}\label{sec_32}  In Lemma \ref{3.3}, we establish the asymptotic expansions of the  functions $F_{d,\ell}$ and  $G_{d,\ell}$ if $\im(z)<0$. 
\begin{lemma}\label{3.3}
For $ z\in\mathbb{C} $  with  $\im(z)<0$, and $N\in\mathbb N_0$, as $t\to 0^+$, we have  that
\begin{align}\label{eqn_Fdfirstsum}
F_{d,\ell}(z;it) &=\left(2\ell^2 t\right)^{-\frac12}e^{\frac{2\pi ijd}{\ell}-\frac{\pi z_0^2}{2\ell^2 t}}\sum_{\substack{n\in\mathbb{Z}\\|n+x_0|\le|y_0|}}e^{-\frac{\pi n^2}{2\ell^2 t}-\frac{\pi z_0 n}{\ell^2 t}-\frac{2\pi ind}{\ell}}
\\ \nonumber &{\hspace{.3in}}+ 
  \sum_{a= 0}^{N} \mathcal{D}_z^{2a}\left(\frac{\zeta^d}{1-\zeta^\ell}\right)\frac{\left(-2\pi t\right)^a}{a!} +O\left(t^{N+1}\right),\\
\label{eqn_Gdfirstsum}
G_{d,\ell}(z;it) &= 2i \left(2\ell^2 t\right)^{-\frac12}e^{-\frac{\pi z_0^2}{2\ell^2 t}}\sum_{n\in\Z\atop |n+x_0|\leq |y_0|}\sin\left(\frac{2\pi d}{\ell}(j-n)\right)e^{-\frac{\pi n^2}{2\ell^2 t}-\frac{\pi z_0 n}{\ell^2 t}}
\\ \nonumber & {\hspace{.3in}} + 
 2i\sum_{a= 0}^N\mathcal{D}_z^{2a}\left(\frac{\sin\left(2\pi dz\right)}{1-\zeta^\ell}\right)\frac{\left(-2\pi t\right)^a }{a!}+ O\left(t^{N+1}\right).
\end{align}
\end{lemma}

\begin{proof}  
  With $$\mathcal{M}_{d, \ell}(z;\tau) :=  \sum_{n\in \mathbb{Z}} \zeta^{\ell n+d} q^{(\ell n+d)^2}, $$  we have  that  
\begin{equation}\label{eqn_Fdsplit}
 F_{d,\ell} (z;\tau) = \mathcal{M}_{d, \ell}(z;\tau) - F_{\ell-d,\ell}(-z;\tau). 
\end{equation}
We analyze the two functions in (\ref{eqn_Fdsplit}) separately.

  Since $\im(-z)>0$,  Lemma 3.1 yields that 
\begin{align}\label{eqn_Fminusd}
 -F_{\ell-d,\ell} (-z;\tau) = \sum_{a\geq 0}  \mathcal{D}_z^{2a} \left(\frac{\zeta^{d}}{1-\zeta^{\ell}} \right)\frac{(2\pi i \tau)^a}{a!}.
\end{align}

Next, we write the function $\mathcal{M}_{d,\ell}$ in terms of the theta function in \eqref{Theta} and use \eqref{thetatrans} to obtain 

\begin{align}
 \label{eqn_Masymp} 
 \mathcal{M}_{d,\ell }(z;it) =\left(2\ell ^2t\right)^{-\frac{1}{2}} e^{\frac{2\pi i jd}{\ell }-\frac{\pi z_0^2}{2\ell ^2t}} \sum_{n\in \Z} e^{-\frac{\pi n^2}{2\ell ^2 t} -\frac{\pi z_0 n}{\ell ^2t} -\frac{2\pi ind}{\ell }}.
\end{align}

 Inserting  (\ref{eqn_Fminusd}) and (\ref{eqn_Masymp}) into  (\ref{eqn_Fdsplit}) we have shown that
$$
F_{d,\ell}(z;it)=\sum_{a\geq 0}\mathcal{D}_z^{2a}\left(\frac{\zeta^d}{1-\zeta^\ell}\right)\frac{\left(-2\pi t\right)^a }{a!}+\left(2\ell^2 t\right)^{-\frac12}e^{\frac{2\pi ijd}{\ell}-\frac{\pi z_0^2}{2\ell^2 t}}\sum_{n\in\Z}e^{-\frac{\pi n^2}{2\ell^2 t}-\frac{\pi z_0 n}{\ell^2 t}-\frac{2\pi ind}{\ell}}.
$$
The claimed expansion for $F_{d,\ell}$ now follows by  ignoring the exponentially small terms.\\
\indent
To prove the result for $G_{d,\ell}$, we use the result just established for $F_{d,\ell}$ together with the definition of $G_{d,\ell}$.
\end{proof}

\indent
Using Lemma \ref{3.3}, we establish the asymptotic behavior of the partial theta functions $F_{d,\ell}$  for $\im(z)<0$ in Corollary \ref{3.4} below.  Unlike the previous case in which $\im(z)>0$, a   more  careful analysis is required.
\begin{corollary}\label{3.4}
We have the following behavior for $\im(z)<0$, as $t\to 0^+$.
\begin{itemize}[leftmargin=*,align=left]
	\item[\rm{(i)}] If $|x_0|>|y_0|$, then
	$$
	F_{d,\ell}(z;it)\sim\frac{\zeta^d}{1-\zeta^\ell}.
	$$
	\item[\rm{(ii)}] If $|x_0|\leq |y_0|$ and $x_0\neq 1/2$, then
	$$
	F_{d,\ell}(z;it)\sim\left(2\ell^2 t\right)^{-\frac12}e^{\frac{2\pi ijd}{\ell}-\frac{\pi z_0^2}{2\ell^2 t}}.
	$$
	\item[\rm{(iii)}] If $x_0=1/2$ and $|y_0|\geq 1/2$, then
	$$
	F_{d,\ell}(z;it)\sim 2\left(2\ell^2 t\right)^{-\frac12}\cos\left(\frac{\pi}{\ell}\left(d+\frac{y_0}{2\ell t}\right)\right)e^{-\frac{\pi}{8\ell^2 t}+\frac{\pi y_0^2}{2\ell^2 t}+\frac{\pi i(2j+1)d}{\ell}}.
	$$
\end{itemize}
\end{corollary}

\begin{proof} One can see that the dominant term  in the sum in \eqref{eqn_Fdfirstsum}  occurs  for $n=0$, and additionally for $n=-1$ if $x_0=1/2$.  Moreover, we obtain   a contribution exactly if  $|x_0|>|y_0|$.  This yields (i).

We next assume $|x_0|\leq |y_0|$. One directly obtains (ii) if $x_0\neq 1/2$.  If $x_0=1/2$, we simplify the exponent of the $n=-1$ term in the sum in (\ref{eqn_Fdfirstsum})  to give the claim (iii).
\end{proof} 
\begin{remark}  Lemma \ref{2.2} together with results in Section 4 (Lemma \ref{lemma0.1} and Corollary {4.2}) gives another method for determining asymptotic behaviors and expansions of the functions $F_{d,\ell}$ in a certain subset of $\{z\in \mathbb C \colon \im(z)>0\}$.  This  provides another ``explanation" of the anti-Stokes lines $|u_0|=|v_0|$.  
\end{remark}
\indent
Finally, we establish the asymptotic main terms of the functions $G_{d,\ell}$   if  $\im(z)<0$. 
\begin{corollary} \label{3.5}
We have the following  behavior for $\im(z)<0$, as $t\to 0^+$:
\begin{itemize}[leftmargin=*,align=left]
	\item[\rm{(i)}] If $|x_0|>|y_0|$ or $( |x_0| \leq |y_0|, \ell \mid 2d$, and $x_0\neq 1/2)$ or $(\ell \mid 2dj$ but $\ell\nmid 2d$, $|y_0|<1-|x_0|$, and $x_0\neq 1/2)$, then
	$$
G_{d,\ell}(z;it)\sim\frac{2i\sin\left(2\pi dz\right)}{1-\zeta^\ell}.
	$$
	\item[\rm{(ii)}] If $\ell\nmid 2dj$, $|x_0|\leq |y_0|$, and $x_0\neq 1/2$, then
	$$
	G_{d,\ell}(z;it)\sim 2i\left(2\ell^2 t\right)^{-\frac12}\sin\left(\frac{2\pi jd}{\ell}\right)e^{-\frac{\pi z_0^2}{2\ell^2 t}}.
	$$
	\item[\rm{(iii)}] If $\ell \mid 2dj$ but $\ell\nmid 2d$, $1-|y_0|\leq |x_0|\leq |y_0|$, and $x_0\neq 1/2$, then
	$$
	G_{d,\ell}(z;it)\sim 2i (-1)^{\frac{2dj}{\ell}} \left(2\ell^2 t\right)^{-\frac12}e^{-\frac{\pi z_0^2}{2\ell^2 t}-\frac{\pi}{2\ell^2t}+\frac{\pi\sgn(x_0)z_0}{\ell^2t}}
	\sgn(x_0)\sin\left(\frac{2\pi d}{\ell}\right)
	$$
	\item[\rm{(iv)}] If $x_0=1/2$ and $|y_0|\geq 1/2$, then
	\begin{align*}
	G_{d,\ell}(z;it)\sim 2\left(2\ell^2 t\right)^{-\frac12}e^{-\frac{\pi}{8\ell^2 t}+\frac{\pi y_0^2}{2\ell^2 t}}\sum_\pm \pm\cos\left(\frac{\pi}{\ell}\left(d\pm\frac{y_0}{2\ell t}\right)\right)e^{\pm\frac{\pi i(2j+1)d}{\ell}}.
	\end{align*}
\end{itemize}
\end{corollary}

\begin{proof}
The only thing that has to be considered is that by applying case (ii) of Corollary \ref{3.4}  to the two functions $F_{d,\ell}$ and $F_{-d,\ell}$ defining $G_{d,\ell}$, the overall contribution could be $0$.  

 If $\ell\nmid 2dj$ this is not the case. If $\ell \mid 2d$, then we use Lemma \ref{3.3}, and find that the entire contribution arising from the first sum on the right-hand side of (\ref{eqn_Gdfirstsum}) vanishes. If $\ell \mid 2dj$ but $\ell\nmid 2d$, {then} we again  use the expansion in Lemma \ref{3.3} and look at the next highest term, which occurs {for} $n=-\sgn(x_0)$. 
\end{proof}

\section{ Asymptotic expansions of $F_{d,\ell}$ and $G_{d,\ell}$ if  $\im(z)= 0$ }\label{sec_zreal}
We are left to consider the asymptotic properties of the  partial theta functions $F_{d,\ell}$ and  the functions   $G_{d,\ell}$ if $\operatorname{Im}(z)=0$. For this, we begin by establishing a slight generalization of Theorem 4.1 of \cite{BK}, namely Lemma \ref{lemma0.1} below, which holds for any $z\in\mathbb C$ satisfying $|z|<1/(4\ell)$.   From this we ultimately deduce asymptotic expansions of the functions $F_{d,\ell}$ and $G_{d,\ell}$ for real $z$ with $|x_0|<1/4$ in Corollary \ref{4.3}.  In Corollary \ref{wholeexp} below we are able to remove the restriction that $|x_0|<1/4$ and provide asymptotic expansions for any $z\in \mathbb R$, after establishing some technical lemmas.

\begin{lemma}\label{lemma0.1}
For $|z|< 1/(4 \ell)$,  $\tau \in \mathbb H,$ $N\in\N_0$, $d\in\Q$, and  $\ell\in\N$, we have that  $$
F_{d,\ell}\left(z;\tau\right)=\sum_{b\geq 0}\frac{(2\pi i\ell z)^b}{b!}\left(\frac{\Gamma\left(\frac{b+1}{2}\right)}{2(-2\pi i\ell^2 \tau)^{\frac{b+1}{2}}}-\sum_{a=0}^N\frac{\left(2\pi i\ell^2 \tau\right)^a}{a!}\frac{B_{2a+b+1}\left(\frac{d}{\ell}\right)}{2a+b+1}\right)+O\left(|\tau|^{N+1}\right).
$$
\end{lemma}
\begin{remark}
For $d>0$, Lemma \ref{lemma0.1} is given in \cite{BK} as Theorem 4.1.
\end{remark}
\begin{proof}
  Choose {any} $r \in \mathbb N$ such that $d':=r\ell +d>0$.  By Lemma \ref{lem_fshift}, we have
\begin{align}\label{Fde2}
F_{d, \ell}(z;\tau)=F_{d', \ell}(z;\tau)+\sum_{a\geq 0} \mathcal{D}_z^{2a}\left(\frac{\zeta^{d}\left(1-\zeta^{\ell r}\right)}{1-\zeta^{\ell}}\right)\frac{(2\pi i\tau)^a}{a!}.  
\end{align}   For the first term in (\ref{Fde2}), we may use the asymptotic expansion in Lemma 4.1 with $d'$ instead of $d$, as it is known to be true from \cite{BK}.    For the second term, we apply Lemma \ref{derivative} twice, giving the claim.
\end{proof}

\indent
We apply Lemma \ref{lemma0.1} to determine  the  asymptotic expansion of the partial theta functions $F_{d,\ell}$. 
\begin{corollary}\label{4.2} 
For $|z|<1/(4\ell)$ and $N\in\N_0$, as $t\to 0^+$, we have that
\begin{multline*}
F_{d,\ell}(z;it)=\frac{1}{2\ell (2t)^{\frac12}}e^{-\frac{\pi z^2}{2t}}\left(1+\operatorname{erf}\left(\frac{\sqrt{\pi}iz}{\sqrt{2t}}\right)\right)\\
+\sum_{a=0}^N \left(\mathcal{D}_z^{2a}\left(\frac{\zeta^d}{1-\zeta^\ell}\right)+\frac{(2a)!}{\ell(2\pi iz)^{2a+1}}\right)\frac{(-2\pi t)^a}{a!}+O\left(t^{N+1}\right).
\end{multline*}
\end{corollary} 
\begin{proof}
 Using Lemma \ref{2.1}, the first term in Lemma \ref{lemma0.1}  for $\tau=it$ evaluates as the first summand in the corollary.
The second term in Lemma \ref{lemma0.1} is
\begin{align*}
-\sum_{a=0}^N \frac{(-2\pi \ell^2t)^a}{a!} \sum_{b\geq 0}\frac{(2\pi i\ell z)^b}{b!}\frac{B_{2a+b+1}\left(\frac{d}{\ell}\right)}{2a+b+1}. 
\end{align*}
The claim follows by applying Lemma \ref{derivative}.
\end{proof}

\indent
We next turn to the question of establishing the asymptotic properties of the  functions $F_{d,\ell}$ and $G_{d,\ell}$ if $z\in\R$.   We first do so in Corollary \ref{4.3} if $|x_0|<1/4$. 
\begin{corollary}\label{4.3} 
For $z=(x_0+j)/\ell \in\R$, where $j\in\Z, \ell \in \mathbb N$,  and $|x_0|<1/4$, for any $N\in\mathbb N_0$, as $t\to 0^+$, we have 
\begin{align*}
F_{d,\ell}\left(z;it\right)&=\frac{e^{\frac{2\pi idj}{\ell}}}{2\ell (2t)^{\frac12}}e^{-\frac{\pi x_0^2}{2\ell^2 t}}\left(1+\operatorname{erf}\left(\frac{\sqrt{\pi}ix_0}{\ell\sqrt{2t}}\right)\right)\\
& \quad+\sum_{a=0}^N \left(\mathcal{D}_z^{2a}\left(\frac{\zeta^d}{1-\zeta^\ell}\right)+e^{\frac{2\pi idj}{\ell}}\frac{(2a)!\ell^{2a}}{(2\pi ix_0)^{2a+1}}\right)\frac{(-2\pi t)^a}{a!}+O\left(t^{N+1}\right),\\
G_{d,\ell}\left(z;it\right)&=\frac{i\sin\left(\frac{2\pi dj}{\ell}\right)}{\ell (2t)^{\frac12}}e^{-\frac{\pi x_0^2}{2\ell^2 t}}\left(1+\operatorname{erf}\left(\frac{\sqrt{\pi}ix_0}{\ell\sqrt{2t}}\right)\right)\\
& \quad +2i\sum_{a=0}^N \left(\mathcal{D}_z^{2a}\left(\frac{\sin\left(2\pi dz\right)}{1-\zeta^\ell}\right)+\frac{\sin\left(\frac{2\pi dj}{\ell}\right)(2a)!\ell^{2a}}{(2\pi ix_0)^{2a+1}}\right)\frac{(-2\pi t)^a}{a!}+O\left(t^{N+1}\right).
\end{align*}

\end{corollary}

\begin{proof}
We have 
$$
F_{d,\ell}(z;\tau)=e^{\frac{2\pi idj}{\ell}}\sum_{n\geq 0}e^{2\pi i\frac{x_0}{\ell}(\ell n+d)}q^{(\ell n+d)^2}=e^{\frac{2\pi idj}{\ell}}F_{d,\ell}\left(\frac{x_0}{\ell};\tau\right).
$$
The expansions for the functions $F_{d,\ell}$ and $G_{d,\ell}$ then follow from Corollary \ref{4.2}.  
\end{proof}

\indent
In order to remove  the  restriction in Corollary \ref{4.3} that $|x_0|<1/4,$ we begin by establishing Lemma \ref{4.4} below, which holds for certain real values of $z$. 
\begin{lemma} \label{4.4} 
If $z=\frac{w_0}{\ell}+\frac{c}{h \ell} \in \mathbb{R}$, $h \geq 2$, ${\rm gcd}(c,h)=1$, and $|w_0| < \frac{1}{4 h}$, then we have for any $N\in \mathbb N_0$, as $t\to 0^+$,  
 
\begin{align*}
F_{d, \ell}(z; it)&=\sum_{a=0}^N\mathcal{D}_z^{2a}\left(\frac{\zeta^d}{1-\zeta^\ell}\right)\frac{(-2\pi t)^a}{a!} +O\left(t^{N+1}\right),
\\
G_{d, \ell}(z; it)&=2i\sum_{a=0}^N\mathcal{D}_z^{2a}\left(\frac{\sin(2  \pi d z)}{1-\zeta^\ell}\right)\frac{(-2\pi t)^a}{a!} +O\left(t^{N+1}\right).
\end{align*}
\end{lemma}

\begin{proof}  
We may write
\begin{equation} \label{modify}
F_{d, \ell}(z;\tau) = e^{\frac{2 \pi i d c}{h \ell}} \sum_{j=0}^{h-1} e^{\frac{2 \pi i c j}{h}} F_{\ell j+d,\ell h }\left(\frac{w_0}{\ell};\tau\right).
\end{equation}

By  (\ref{modify}), Corollary \ref{4.2},  and the fact that $\sum_{j=0}^{h-1} e^{\frac{2 \pi i c j}{h}}=0$, we obtain
\begin{align*}
& F_{d,\ell}(z;it)=- e^{\frac{2 \pi i d c}{h \ell}} \sum_{j=0}^{h-1} e^{\frac{2 \pi i c j}{h}} \sum_{a=0}^N \left(\mathcal{D}_{\frac{w_0}{\ell}}^{2a}\left(\frac{e^{2\pi i (\ell j +d) \frac{w_0}{\ell}}}{e^{2\pi i \ell h \frac{w_0}{\ell}}-1}\right)-\frac{(2a)!}{\ell h \left(2\pi i \frac{w_0}{\ell}\right)^{2a+1}}\right) \frac{(-2\pi t)^a}{a!} +O\left(t^{N+1}\right)  \\  
& = -  \sum_{a=0}^N \mathcal{D}_{z}^{2a}\left(\sum_{j=0}^{h-1} e^{2 \pi i \left(\frac{dc}{h  \ell}+\frac{c j}{h} \right) } \frac{e^{2\pi i (\ell j +d) \frac{w_0}{\ell}}}{e^{2\pi i h w_0}-1}\right)\frac{(-2\pi t)^a}{  a!} +O\left(t^{N+1}\right) \\
& = -  \sum_{a=0}^N  \mathcal{D}_{z}^{2a}\left(\frac{e^{\frac{2 \pi i d c}{h \ell}+\frac{2 \pi i d w_0}{\ell}}}{e^{2 \pi i h w_0} -1} \sum_{j=0}^{h-1} e^{2 \pi i \left(\frac{c}{h}+w_0\right)j}  \right) \frac{(-2\pi  t)^a }{ a!}+O\left(t^{N+1}\right) \\ 
& = -  \sum_{a=0}^N  \mathcal{D}_{z}^{2a}\left(\frac{e^{2 \pi i d\left(\frac{c}{h \ell}+\frac{w_0}{\ell}\right)}}{e^{2 \pi i \ell \left(\frac{c}{h \ell}+\frac{w_0}{\ell} \right)}-1}\right)\frac{(-2\pi t)^a}{a!}+O\left(t^{N+1} \right).
\end{align*} 
 From this, we may conclude the claim.
\end{proof}

\indent
In order to fully remove the restriction that $|x_0|<1/4$ in Corollary \ref{4.3}, in addition to Lemma \ref{4.4},  we need the following technical result.
\begin{lemma} \label{open.cover} 
We have $\Omega=\mathbb{R} \setminus  \mathbb{Z}$, where  $$\Omega:=\left\{ w_0+\frac{c}{h} :  \frac{c}{h} \in \mathbb{Q}\setminus\{0\},  \gcd(c,h)=1 , h \geq 2, w_0 \in \mathbb R, |w_0| < \frac{1}{4 h} \right\}.$$
\end{lemma}

\begin{proof}  
A short calculation reveals that the sets $$D_{\frac{c}{h}}:=\left\{w_0 \in \mathbb R : \left|w_0 -\frac{c}{h}\right| <\frac{1}{4h }\right\}, $$ where $c/h \in \mathbb{Q}\setminus\{0\}$, cover the interval $(0,1)$, and thus, $\mathbb R \setminus \mathbb Z \subseteq \Omega.$  By considering the cases $0<w_0 < \frac{1}{4h}$ and $-\frac{1}{4h} < w_0 < 0$ separately, it is not difficult to see that $\Omega$ contains no integer.
\end{proof}

We are now able to establish the general asymptotic expansions of the functions $F_{d,\ell}$ and $G_{d,\ell}$ if $\im(z)=0$. 

\begin{corollary}  \label{wholeexp}
For $z\in \mathbb{R}$, as $t\to 0^+$, for any $N\in\mathbb N_0$, the following are true. \\
\rm{(i)}  If $z\in \frac{1}{\ell}\mathbb Z$, then we have that  
\begin{align*}
F_{d,\ell}\left(z; it\right)& =\frac{e^{2\pi i d z}}{2\ell\left(2t\right)^{\frac{1}{2}}}-e^{2\pi i d z}\sum_{a=0}^{N} \frac{ B_{2a+1}\left(\frac{d}{\ell}\right)\left(-2 \pi \ell^2 t\right)^a}{(2a+1)a!}+O\left(t^{N+1}\right), \\
G_{d,\ell}\left(z; it\right)& =  \frac{i\sin(2\pi dz)}{ \ell\left(2t\right)^{\frac{1}{2}}}-\sum_{a=0}^{N} \left(e^{2\pi idz}B_{2b+1}\left(\frac{d}{\ell}\right) - e^{-2\pi idz}B_{2a+1}\left(-\frac{d}{\ell}\right)\right)\frac{\left(-2 \pi \ell^2 t\right)^a }{a!(2a+1)} + O\left(t^{N+1}\right).
\end{align*} 
 \noindent \rm{(ii)}  If $z\in \mathbb R \setminus  \frac{1}{\ell}\mathbb Z$, then we have that 
 \begin{align*}
 F_{d,\ell}\left(z; it\right)& =  \sum_{a=0}^N \mathcal D_z^{2a} \left(\frac{e^{2 \pi i d z}}{1-\zeta^\ell}\right)\frac{(-2\pi t)^a}{a!} +O\left(t^{N+1}\right), \\
 G_{d,\ell}\left(z; it\right)& =  2i\sum_{a=0}^N\mathcal D_z^{2a} \left(\frac{\sin(2 \pi   d z)}{1-\zeta^\ell}\right) \frac{(-2\pi t)^a}{a!} +O\left(t^{N+1}\right). 
 \end{align*}
\end{corollary}
\begin{proof} 
\ \\  \rm{(i)}  In this case, we may apply Corollary \ref{4.3} with $x_0=0$.  We first note that $\erf(0)=0$.  Next, in order to evaluate the sum on $b$ which appears, we use (\ref{eqn_dber0}) with $z\mapsto x_0, \ell \mapsto 1,$ and $d\mapsto d/\ell$.  
This gives the expression for $F_{d,\ell}$ in case (i) of Corollary \ref{wholeexp}.  The expression for $G_{d,\ell}$ immediately follows.
\\ \rm{(ii)} In this case, we have that $z\in \mathbb R \setminus \frac{1}{\ell} \mathbb Z$, hence $\ell z \in \mathbb R \setminus \mathbb Z$, thus, we may apply  Lemma \ref{open.cover}  to $\ell z$, to  deduce that we may write $ z = w_0/\ell + c/(h\ell)$ for some $c \in \mathbb Z, h\in \mathbb N$ such that $\gcd(c,h)=1$ and $h\geq 2$, and some $w_0\in\mathbb R$ satisfying $|w_0|\leq 1/(4h)$.    We then apply Lemma \ref{4.4} to obtain the expressions for $F_{d,\ell}$ and $G_{d,\ell}$ given in (ii) of Corollary \ref{wholeexp}.
\end{proof}

From Corollary \ref{wholeexp}, we deduce the following asymptotic behavior if $\im(z)=0$.

\begin{corollary}\label{cor_fgreal}  As $t\to 0^+$, the following are true.
\medskip \\
{\rm{(i)}}  If $z\in\frac{1}{\ell}\Z$, we have  that  
$$\displaystyle F_{d,\ell}(z;it)\sim\frac{\zeta^d}{2\ell(2t)^{\frac12}}.$$
\medskip \\
{\rm{(ii)}} If $z\in \mathbb R \setminus \frac{1}{\ell}\mathbb Z$, we have that  
$$ \displaystyle F_{d,\ell}(z;it)\sim\frac{\zeta^d}{1-\zeta^\ell}.$$
\medskip \\
{\rm{(iii)}}
 If $z\in\frac{1}{\ell}\Z$ and $2 d z \not \in \mathbb Z$, we have that 
$$  G_{d,\ell}(z;it)\sim\frac{ i \sin(2 \pi d z) }{\ell(2t)^{\frac12}}. $$ 
\medskip \\
{\rm{(iv)}}   If $z\in \mathbb R \setminus \frac{1}{\ell}\mathbb Z$ and $2dz \not \in \mathbb Z$, we have that  $$ G_{d,\ell}(z;it)\sim\frac{2 i \sin(2 \pi d z)}{1-\zeta^\ell}.
	$$ 
\medskip \\
{\rm{(v)}} If $z\in\frac{1}{\ell}\Z$ and $2 d z \in \mathbb Z$, we have that
$$G_{d,\ell}(z;it) \sim (-1)^{2dz+1} \frac{2d}{\ell}.$$ 
\medskip \\
{\rm{(vi)}}  If $z\in\mathbb R \setminus \frac{1}{\ell}\Z$ and $2 d z   \in \mathbb Z$, we have that
$$G_{d,\ell}(z;it) \sim (-1)^{2dz + 1} \frac{ 8\pi \ell d \zeta^\ell}{(1-\zeta^\ell)^2}  \ t.$$ 
\end{corollary}
\begin{remark} The functions given on the right-hand sides of the displayed asymptotics in parts (v) and (vi) are  equal to zero if and only if $d=0$, in which case the functions $G_{0,\ell}$ appearing on the left-hand sides are identically equal to zero.
\end{remark}
 
\section{Asymptotic behavior of $F_{d,\ell}$ and $G_{d,\ell}$} \label{sec_fgsummary}  
We are now able to prove Theorem \ref{Fasex} and Corollary \ref{5.2}, using results established in Sections \ref{sec_aux}, \ref{sec_sec3}, and \ref{sec_zreal}.   In particular, Theorem \ref{Fasex} follows  from Lemma \ref{3.1}, Lemma \ref{3.3}, and Corollary \ref{wholeexp}.  Corollary \ref{5.2} can be concluded from Corollaries \ref{cor_fgasypos}, \ref{3.4}, and \ref{cor_fgreal}.  We  establish results analogous to Theorem \ref{Fasex} and Corollary \ref{5.2} for  the  functions $G_{d,\ell}$ in Theorem \ref{Gasex} and Corollary \ref{5.4} below.  These follow in exactly the same way, with the exception of using Corollary \ref{3.5} instead of  Corollary \ref{3.4}.  {As we have seen previously, these behavior   depend on where in $\mathbb C$ the Jacobi variable $z$ is located.  We begin with Theorem \ref{Gasex}, which gives  asymptotic expansions for the functions $G_{d,\ell}$ as $t\to 0^+$.  } 
\begin{theorem}\label{Gasex}  We have the following behavior, as $t\to 0^+$, for any $N\in\mathbb N_0.$
\begin{itemize}[leftmargin=*,align=left]
\item[\rm{(i)}]
If  $\imm(z)>0$ or  $\left(\imm(z)<0 \text{ and } |x_0|>|y_0|\right)$,  \text{ or } $\left(z\in\mathbb{R}\setminus\frac{1}{\ell}\mathbb{Z}\right)$, then 
\begin{align*}
G_{d,\ell}\left(z;it\right)= 2i\sum_{a=0}^{N}\mathcal{D}_z^{2a}\left(\frac{\sin(2\pi dz)}{1-\zeta^\ell}\right)\frac{(-2\pi t)^a}{a!} + O\left(t^{N+1}\right).\\
\end{align*}
\item[\rm{(ii)}]
If $\imm(z)<0\text{ and }|x_0|\le|y_0|$, then
\begin{align*}
  G_{d,\ell}\left(z;it\right)=& 2i\left(2\ell^2t\right)^{-\frac12}e^{-\frac{\pi z_0^2}{2\ell^2t}}\sum_{\substack{n\in\mathbb{Z}\\|n+x_0|\le|y_0|}}e^{-\frac{\pi n^2}{2\ell^2t}-\frac{\pi z_0n}{\ell t}}\sin\left(\frac{2\pi d}{\ell}(j-n)\right)
\\
&{\hspace{.6in}}+ 2i\sum_{a=0}^{N}\mathcal{D}_z^{2a}\left(\frac{\sin(2\pi dz)}{1-\zeta^\ell}\right)\frac{(-2\pi t)^a}{a!}+O\left(t^{N+1}\right). \\
\end{align*}
\item[\rm{(iii)}] If $z\in\frac{1}{\ell}\mathbb{Z}$, then 
\begin{equation*}
G_{d,\ell}(z;it)=\frac{i\sin(2\pi dz)}{\ell(2t)^{\frac{1}{2}}}-\sum_{a=0}^{N}\left(\zeta^d B_{2a+1}\left(\frac{d}{\ell}\right)-\zeta^{-d}B_{2a+1}\left(-\frac{d}{\ell}\right)\right) \frac{\left(-2\pi \ell^2t\right)^a}{a!(2a+1)}+ O\left(t^{N+1}\right).
\end{equation*}
\end{itemize}
\end{theorem}
Next we give the asymptotic behavior of the functions $G_{d,\ell}$  as $t\to 0^+$.  There are seven different cases in this result, dependent on the location of $z$ in $\mathbb C$.

\begin{remark}
By using the previous theorem, we can now easily compute the asymptotic expansion of 
$$\widetilde{G}_{d,\ell}(z;\tau):=G_{d,\ell}(z;\tau)+\zeta^\ell q^{d^2}.$$
This can be used to recover several asymptotic formulas previously obtained in \cite{BerndtKim}, \cite{BM} and other papers for certain special values of $z$, $d$, and $\ell$. 
\end{remark}

\begin{corollary}\label{5.4}
We have the following behavior as $t\to 0^+$.
\begin{itemize}[leftmargin=*,align=left]
\item[\rm{(i)}]  If $\imm(z)>0$ or $(\imm(z)<0$ and $|x_0|>|y_0|)$ or $(\imm(z)<0$, $|x_0|\le|y_0|,~\ell \mid 2d,$ and $x_0\ne1/2)$ or $(\imm(z)<0$, $\ell \mid 2dj$ but $\ell\nmid 2d$, $|y_0|<1-|x_0|$, and $x_0\ne1/2)$ or 
$(z\in\mathbb{R}\setminus\frac{1}{\ell}\mathbb{Z}$ and $2dz \not\in\mathbb Z)$, 
then
$$
G_{d,\ell}\left(z;it\right)\sim\frac{2i\sin\left(2\pi dz\right)}{1-\zeta^\ell}.
$$

\item[\rm{(ii)}] If $\imm(z)<0$, $|x_0|\leq|y_0|$, $x_0\ne1/2$, and $\ell \nmid 2dj$, then
$$
G_{d,\ell}\left(z;it\right)\sim2i\left(2\ell^2t\right)^{-\frac12}\sin\left(\frac{2\pi jd}{\ell}\right)e^{-\frac{\pi z_0^2}{2\ell^2t}}.
$$
\item[\rm{(iii)}] If $\imm(z)<0,~\ell \mid 2dj$ but $\ell\nmid 2d$ and $1-|y_0|\leq|x_0|\leq|y_0|$ and $x_0\neq 1/2$, then 
	$$
	G_{d,\ell}(z;it)\sim 2i (-1)^{\frac{2dj}{\ell}} \left(2\ell^2 t\right)^{-\frac12}e^{-\frac{\pi z_0^2}{2\ell^2 t}-\frac{\pi}{2\ell^2t}+\frac{\pi\sgn(x_0)z_0}{\ell^2t}}\sgn(x_0)\sin\left(\frac{2\pi d}{\ell}\right)	
	$$

\item[\rm{(iv)}] If $\imm(z)<0$, $x_0=1/2$ and $|y_0|\geq 1/2$, then
	\begin{align*} 
	G_{d,\ell}(z;it)\sim 2\left(2\ell^2 t\right)^{-\frac12}e^{-\frac{\pi}{8\ell^2 t}+\frac{\pi y_0^2}{2\ell^2 t}}\sum_\pm \pm\cos\left(\frac{\pi}{\ell}\left(d\pm\frac{y_0}{2\ell t}\right)\right)e^{\pm\frac{\pi i(2j+1)d}{\ell}}.
	\end{align*}

\item[\rm{(v)}]
If $z\in\frac{1}{\ell}\mathbb{Z}$ with $2dz\not\in\Z$, then we have
\begin{equation*}
G_{d, \ell}(z; it)\sim	i\left(2\ell^2t\right)^{-\frac12}\sin(2\pi dz).
\end{equation*}
\item[\rm{(vi)}]
If $z\in\frac{1}{\ell}\Z$ and $2dz\in\Z$, then
$$
G_{d, \ell}(z; it)\sim (-1)^{2dz + 1}\frac{2d}{\ell}. 
$$
\item[\rm{(vii)}] If $z\in\mathbb R \setminus \frac{1}{\ell}\Z$ and $2dz\in\Z$, then 
$$G_{d,\ell}(z;it) \sim (-1)^{2dz + 1} \frac{ 8\pi \ell d \zeta^\ell}{(1-\zeta^\ell)^2}  \ t.$$
\end{itemize}
\end{corollary}

\section{Regularized characters of singlet algebra modules}\label{sec_singlet}
In this section we apply the results from  the previous sections to study asymptotic properties of characters of the $(1,p)$-singlet vertex operator algebra, $p \in \mathbb{N}_{\geq2}$, and their 
quantum dimensions.
We do not recall the definition of the $(1,p)$-singlet vertex algebra here; instead we refer the reader to \cite{CM}.  In this paper we are only interested in regularized characters of irreducible modules, whose explicit formulae we recall next.

In vertex algebra theory, it is customary to use $\ch[M](\tau)$ to denote  the character (or modified graded dimension) of a $V$-module $M$. By definition,
$$\ch[M](\tau)={\rm tr}_M q^{L(0)-\frac{c}{24}},$$
where $L(0)$ is the degree operator (acting semisimply) of $M$ and $c \in \mathbb{C}$ is the central charge.
A regularized character of $M$ is simply a function depending on a complex variable $\varepsilon$,   denoted by $\ch[M^\varepsilon](\tau)$, such that
$$\lim_{\varepsilon \to 0}  \ch[M^\varepsilon](\tau)=\ch[M](\tau).$$
Although there are many different ways to introduce a regularization, for the singlet vertex algebra this can be done canonically via resolutions in terms of Fock modules (again for details see \cite{CM}). 
One of the upshots of the regularization in \cite{CM}  is a one-to-one correspondence among irreducible modules 
and their regularized characters. 

As we already mentioned in the introduction, the $(1,p)$-singlet vertex algebra admits two types of irreducible characters: {atypical} and {typical}. 
{Here we are only interested in regularized characters. Typical (regularized) characters are given by 
\begin{equation}
\label{Fl}
\ch[F_{\lambda}^\varepsilon] (\tau):=\frac{e^{2\pi \varepsilon \left(\lambda-\frac{\alpha_0}{2}\right)}q^{\frac{1}{2} \left(\lambda-\frac{\alpha_0}{2}\right)^2}}{\eta(\tau)},
\end{equation} 
where $\lambda \in \mathbb{C}$, $\alpha_0:=\sqrt{2p}-\sqrt{2/p}$ and $\eta(\tau):=q^{1/24} \prod_{n\geq1} (1-q^n)$ is Dedekind's $\eta$-function. On the other hand, atypical characters are given by  \cite{CM}:  \begin{eqnarray}
&&  \ch [M_{r,s}^\varepsilon](\tau)  \nonumber \\
&&  \label{atypical} =\frac{1}{\eta (\tau)} \sum_{n\geq 0} \left( e^{\frac{2\pi \varepsilon}{\sqrt{2p}} \left(2pn-s-pr+2p\right)} 
 q^{\frac{1}{4p} \left(2pn -s-pr +2p\right)^2} -e^{\frac{2\pi \varepsilon}{\sqrt{2p}} \left(2pn+s-pr+2p\right)} q^{\frac{1}{4p} \left(2pn +s-pr+2p \right)^2} \right),
\end{eqnarray}
where $r \in \mathbb{Z}$, and $ 1 \leq s \leq p-1$.
We stress that in particular, $M_{1,1}$ is the $(1,p)$-singlet vertex algebra \cite{CM}. We note that \cite{CMW} used a slightly different $\varepsilon$-parametrization; $\ch [M_{r,s}^\varepsilon]$ in \cite{CM} is precisely $\ch [M_{r,s}^{-\varepsilon}]$ in \cite{CMW}.
 We normalize the atypical characters $\ch[M_{r,s}^\varepsilon]$ by defining the functions }
\[
 \mathcal{C}_{r,s} (\varepsilon ;\tau) := \eta (\tau) \operatorname{ch}\left[M_{r,s}^\varepsilon \right] (\tau).
\]
{It is not difficult to see that we may decompose the normalized characters $\mathcal C_{r,s}$ as a difference of functions defined using the Jacobi partial theta functions $F_{d,\ell}$ as follows:}
\[
\mathcal{C}_{r,s} (\varepsilon ;\tau) =  F_{2p-s-pr, 2p} \left( \frac{-i\varepsilon}{\sqrt{2p}} ;\frac{\tau}{4p} \right) -F_{2p+s-pr, 2p} \left( \frac{-i\varepsilon}{\sqrt{2p}} ;\frac{\tau}{4p} \right) .
\]

Our next results {further relate the} regularized atypical characters to some of the functions studied earlier.
\begin{lemma}\label{5.1} 
For $\varepsilon\in\mathbb{C},$ we have that  
$$
\mathcal{C}_{r,s}(\varepsilon; \tau)=-G_{s+p(r-2), 2p}\left(-\frac{i\varepsilon}{\sqrt{2p}};\frac{\tau}{4p}\right)
+\sum_{a\geq 0} \left[\mathcal{D}_z^{2a}\left(\frac{\zeta^{s+(r-2)p}\left(1-\zeta^{-2p(r-2)}\right)}{1-\zeta^{2p}}\right)\right]_{z=-\frac{i\varepsilon}{\sqrt{2p}}}\frac{\left(\frac{\pi i\tau}{2p}\right)^a}{a!}.
$$
\end{lemma}
\begin{proof}
The claim follows by applying Lemma \ref{lem_fshift}. 
\end{proof}

For the remainder of this section, as introduced in the introduction,  we write $\varepsilon=(\varepsilon_0+ik)/\sqrt{2p}$, with $k \in \mathbb{Z}$, $\varepsilon_0 = u_0 + i v_0$ with $u_0, v_0 \in \mathbb R,$ and $-1/2 < v_0 \leq 1/2 $.
Using Lemma \ref{5.1} and results from Section \ref{sec_fgsummary}, we establish the asymptotic expansions of the functions $\mathcal C_{r,s}$ in Theorem \ref{6.2}.
\begin{theorem}\label{6.2}  We have the following behavior,  as $t\to 0^+$, for any $N\in\mathbb N_0$.
\begin{itemize}[leftmargin=*,align=left]
\item[\rm{(i)}] If $\re(\varepsilon)<0$ or $(\re(\varepsilon)>0$ and $|v_0|>|u_0|)$ or $(u_0=0$ and $v_0\neq 0),$
then 
$$
\mathcal{C}_{r,s}(\varepsilon;it)=\sum_{a=0}^N \mathcal{D}_\varepsilon^{2a}\left(\frac{e^{\pi\sqrt{2p}(1-r)\varepsilon}\sinh\left(\sqrt{\frac{2}{p}}\pi s\varepsilon\right)}{\sinh\left(\sqrt{2p}\pi\varepsilon\right)}\right) \frac{(\pi t)^a}{a!}+ O\left(t^{N+1}\right).
$$
\item[\rm{(ii)}] If $\re(\varepsilon)>0$, and $|v_0| \leq |u_0|$, then  
\begin{multline*}
\mathcal{C}_{r,s}(\varepsilon;it)= -2i(2p t)^{-\frac12} e^{\frac{\pi\varepsilon_0^2}{2p t}}\sum_{\substack{n\in\mathbb{Z} \\ |n+v_0|\le|u_0|}}(-1)^{r(k+n)}\sin\left(\frac{\pi s(k-n)}{p}\right)e^{-\frac{\pi n^2}{2p t}+\frac{2\pi i\varepsilon_0 n}{t}}\\
+\sum_{a=0}^N\mathcal{D}_\varepsilon^{2a}\left(\frac{e^{\pi\sqrt{2p}(1-r)\varepsilon}\sinh\left(\sqrt{\frac{2}{p}}\pi s\varepsilon\right)}{\sinh\left(\sqrt{2p}\pi\varepsilon\right)}\right)\frac{(\pi t)^a}{a!} + O\left(t^{N+1}\right).
\end{multline*}
\item[\rm{(iii)}] If $\varepsilon_0 = 0$, then 
\begin{multline*}
 \mathcal{C}_{r,s}(\varepsilon;it) = - i(-1)^{rk} \sin \left( \frac{\pi ks}{p}\right)(2pt)^{-\frac{1}{2}} \\ + {(-1)^{kr}}\sum_{a=0}^N {\left(e^{\frac{\pi i s k}{p}}   B_{2a+1} \left( \frac{s+p(r-2)}{2p} \right) - e^{-\frac{\pi i s k}{p}}  B_{2a+1} \left( -\frac{s+p(r-2)}{2p} \right)\right)}\frac{(-2\pi pt)^a}{{(2a+1)}a!} \\
 - \sum_{a= 0}^N \mathcal{D}_\varepsilon^{2a} 
\left(  e^{2 \pi \varepsilon \frac{(s-p)}{\sqrt{2p}}}   \frac{\sinh(\sqrt{2p} (r-2)\pi \varepsilon)}{ \sinh(\sqrt{2p} \pi \varepsilon) }\right)\frac{(\pi t)^a}{a!} + O\left(t^{N+1}\right).
\end{multline*} 
\end{itemize}
\end{theorem}
\begin{proof}
We use Lemma \ref{5.1} and Theorem \ref{Gasex} with $z=- i\varepsilon / \sqrt{2p}$,  and $\tau\mapsto \tau/(4p)$.  Recalling that $z=(z_0+j)/\ell$ and $\varepsilon=(\varepsilon_0 + ik)/\sqrt{2p}$, we thus take 
$\ell  = 2p, \ j=k, \ z_0 = -i\varepsilon_0, \ x_0=v_0, \text{ and } y_0=-u_0.$ The conditions from Theorem \ref{Gasex} then translate as follows:

\begin{align*}
\mathrm{Im}(z)>0 \Leftrightarrow   \re(\varepsilon)<0,\quad |x_0|>|y_0| \Leftrightarrow  |v_0|> |u_0|,\quad 
 z\in\R\setminus \tfrac{1}{\ell}\Z  \Leftrightarrow  u_0=0 \ {\rm and } \  v_0 \neq 0.\end{align*}

Parts (i) and (ii) follow by combining Lemma \ref{5.1} and Theorem \ref{Gasex} (i) and (ii), respectively.  To prove (iii), we proceed similarly and combine 
Lemma \ref{5.1} and Theorem \ref{Gasex} (iii). The first term arising from (iii) of Theorem \ref{Gasex}  simplifies to be the first term of the statement. For the second term arising from (iii) of Lemma \ref{5.1} we use a direct substitution. Finally the remaining term from Lemma \ref{5.1} yields the third term after a change of variables.
\end{proof}

\begin{remark}
Using the fact that $\mathcal{C}_{r,s} = F_{2p-s-pr,2p} - F_{2p+s-pr,2p}$, one may alternatively establish asymptotic results for the functions $\mathcal{C}_{r,s}$ using the asymptotic results for the functions $F_{d,\ell}$ obtained in the previous sections.
\end{remark}

\begin{remark} Define $\mathcal{C}^{(m)}_{r,s}(\varepsilon;\tau):=\mathcal{D}_{\varepsilon}^m \left( \mathcal{C}_{r,s}(\varepsilon;\tau) \right)$.
For $m=1$, these and related ``weight $3/2$" false theta functions were studied in \cite{BM,CMW} in connection to atypical characters of the $(p,p')$-singlet vertex algebras. Their asymptotic expansion can be computed from the previous theorem by differentiating the asymptotic expansion in (i)-(iii) term by term.
\end{remark}

\indent
Using Theorem \ref{6.2} as well as some earlier results, we establish the asymptotic behavior of the normalized characters $\mathcal C_{r,s}$ in Corollary \ref{6.3}. Recall that $1\leq s\leq p-1$; for this reason, if we encounter the hypothesis  $p\nmid s$ in the proof of Corollary \ref{6.3} below, we refrain from writing it down.  Similarly, $p\mid s$ can not occur, so if tabulating previous results to formulate Corollary \ref{6.3} we (must) omit such cases.
\begin{corollary} \label{6.3}
We have the following behavior, as $t\rightarrow 0^+.$ 
\begin{itemize}[leftmargin=*,align=left] 
\item[\rm{(i)}]  If $\ree(\varepsilon)<0$  or $(\ree(\varepsilon)>0 \text{ and } |v_0|>|u_0|)$ or $(\ree(\varepsilon)>0, p \mid ks, ~|u_0| <1 -|v_0|, \text{ and }  v_0 \ne1/2)$ or $(u_0 = 0, v_0\neq 0,\text{ and }s(v_0+k)/p\not\in \mathbb Z)$, 
then
$$
\mathcal{C}_{r,s}\left(\varepsilon;it\right)\sim e^{\pi \varepsilon \sqrt{2p}(1-r)} \frac{\sinh\left(\frac{\sqrt{2}\pi s\varepsilon}{\sqrt{p}}\right)}{\sinh\left(\sqrt{2p}\pi\varepsilon\right)}.
$$

\item[\rm{(ii)}] If $\ree(\varepsilon)>0$, $|v_0| \leq |u_0|,~v_0\ne1/2$, and $p\nmid ks$, then
$$
\mathcal{C}_{r,s}\left(\varepsilon;it\right)\sim - 2i(-1)^{kr} (2pt)^{-\frac12}\sin\left(\frac{\pi ks}{p}\right)e^{\frac{\pi \varepsilon_0^2}{2p t}}.
$$
\item[\rm{(iii)}] If $\ree(\varepsilon)>0,~p \mid ks,~1-|u_0| \leq |v_0|\leq |u_0|,\text{ and }v_0\ne1/2$, then $$ 
\mathcal{C}_{r,s}\left(\varepsilon;it\right)\sim  - 2i (-1)^{r(k+1)+\frac{sk}{p}} (2pt)^{-\frac12}e^{\frac{\pi\varepsilon_0^2}{2p t}-\frac{\pi}{2p t}-\frac{\pi i\sgn\left(v_0\right)\varepsilon_0}{p t}} \sgn(v_0)\sin\left(\frac{\pi s}{p}\right).$$
\item[\rm{(iv)}] If $\ree(\varepsilon)  >0$, $v_0 =1/2$, and $|u_0| \geq 1/2$, then 
$$
\mathcal{C}_{r,s} (\varepsilon ;it) \sim - 2(2pt)^{-\frac12} e^{-\frac{\pi}{8pt} + \frac{\pi u_0^2}{2pt}} \sum_\pm \pm\cos\left(\frac{\pi}{2p}\left(s+pr\mp\frac{u_0}{t}\right)\right)e^{\pm\frac{\pi i(2k+1)(s+pr)}{2p}}.
$$
\item[\rm{(v)}] If  $\varepsilon_0=0$ and $p \nmid ks$, then
\[
\mathcal{C}_{r,s} (\varepsilon ; it) \sim -i (-1)^{kr}  (2p t)^{-\frac12} \sin \left(\frac{\pi ks}{p} \right).
\]    
\item[\rm{(vi)}]    If   $\varepsilon_0=0$ and $p \mid ks$, then 
\[
\mathcal{C}_{r,s} (\varepsilon ; it) \sim  (-1)^{kr+\frac{ks}{p}}\frac{s}{p}.
\]  
\item[\rm{(vii)}]  If  $u_0=0,v_0 \neq 0 $, and $s(v_0+k)/p\in \mathbb Z,$ then
\[
\mathcal{C}_{r,s} (\varepsilon ; it) \sim 2\pi s (-1)^{\frac{s(v_0+k)}{p} +kr}  e^{-\pi i r v_0} \frac{\left(r-1-i\cot\left(\pi v_0  \right)\right)}{1-e^{-2\pi i v_0}} \ t.
\]
\end{itemize}
\end{corollary}

\begin{proof}  
\ \\  (i)  Under the first, second, and fourth set of hypotheses given, we use Theorem \ref{6.2}.  Under the third set of hypotheses given, we apply  Lemma \ref{5.1} and appeal to the fourth set of hypotheses in Corollary \ref{5.4} (i).  The main term in the asymptotic expansion then follows from Theorem \ref{6.2}, taking $a=0$.\\
(ii)   Corollary \ref{5.2}  directly yields  the claim.\\
\noindent (iii)  The claim follows from Corollary \ref{5.2}, after simplifying.  \\
{\noindent (iv) and (v)  follow from Corollary \ref{5.4}, Lemma \ref{5.1}, and Theorem \ref{6.2}.} \\
\noindent (vi)  The claim follows directly from Corollary 5.2 and Lemma 6.1 since 
$$
\lim_{t\to 0}\frac{1-e^{-2pt(r-2)}}{1-e^{2pt}}=(2-r).
$$
{\noindent (vii) The claim follows from Theorem \ref{6.2} (i), using the $a=1$ term of the sum given.}
\end{proof} 
\indent
We finally prove Theorem \ref{qdim}, which establishes the ($\varepsilon$-regularized) quantum dimensions of the singlet algebra modules. Recall from (\ref{quantum-dim}), that  
\begin{equation*}
{\rm qdim}[M^\varepsilon_{r,s}]=\lim_{t \to 0^+} \frac{\ch[M_{r,s}^\varepsilon] (it)}{\ch[M_{1,1}^\varepsilon] (it)} \quad\qquad\text{and}\qquad\quad
{\rm qdim}[F^\varepsilon_\lambda]=\lim_{t \to 0^+} \frac{\ch[F_\lambda^\varepsilon] (it)}{\ch[M_{1,1}^\varepsilon] (it)}.
\end{equation*} 
\begin{proof}[Proof of Theorem \ref{qdim}] For ${\rm qdim}[M^\varepsilon_{r,s}], $ all of the statements follow from Corollary \ref{6.3} in a straightforward manner, except for certain parts of (i), (ii), and (iv),   which we now elaborate upon.

For part (i), we first note that in establishing (\ref{qdimM13i}) for $\text{qdim}[M^\varepsilon_{r,s}]$ in the case $(\ree(\varepsilon)>0, p|k, |u_0|<1-|v_0|, \text{and } v_0\neq 1/2)$, we apply Corollary \ref{6.3} (i) twice, using that $p \mid k$ implies $p \mid ks$, to establish the asymptotic behaviors of $\mathcal C_{r,s}$ and $\mathcal C_{1,1}$, from which the result follows  in this case.   Next, we consider the two cases $(\operatorname{Re}(\varepsilon)>0,$ and $| v_0| >| u_0|),$ and $\operatorname{Re}(\varepsilon)<0$.  The claimed results for $\text{qdim}[M^\varepsilon_{r,s}]$ in these cases again follow by applying Corollary \ref{6.3} (i) twice, once to 
$\mathcal C_{r,s}$ and once to $\mathcal C_{1,1}$. Turning to the last set of hypotheses $(u_0=0$ and $v_0\neq 0)$, the proof splits into two cases:  $s(k+v_0)/p \not \in \mathbb Z$ and $s(k+v_0)/p \in \mathbb Z$.  In the former case, we again apply Corollary \ref{6.3} (i)  to establish the asymptotic behaviors of $\mathcal C_{r,s}$ and $\mathcal C_{1,1}$.  For this we additionally require that  $(v_0+k)/p\not \in \mathbb Z$, but since $0<\vert v_0\vert \leq 1/2$, this condition always holds.    In the latter case $(s(k+v_0)/p \in \mathbb Z)$ we apply Corollary \ref{6.3} (vii) for $s$, and Corollary \ref{6.3} (i) for $s=1,$ to establish the asymptotic behaviors of $\mathcal C_{r,s}$ and $\mathcal C_{1,1}$ respectively.  We find that $\text{qdim}[M^\varepsilon_{r,s}]=0$ in this case, which agrees with   (\ref{qdimM13i}) under the hypotheses given.   

 To establish part (ii) for $\text{qdim}[M_{r,s}^\varepsilon]$,  we distinguish three subcases of the hypotheses $(\ree(\varepsilon)>0$, $p\nmid k$,  $|v_0| \leq |u_0|,$ and $v_0\ne1/2)$. 
Firstly,  if we additionally have that $p\nmid ks$, then we may apply Corollary \ref{6.3} (ii) twice to obtain the claim. Next, if in addition to the original hypotheses we have that $p\mid ks$ and $|u_0|<1-|v_0|$, then we apply Corollary \ref{6.3} (i) for $s$ and part (ii) for $s=1$ yielding $\operatorname{qdim}\left[M_{r,s}^\varepsilon\right]=0$ which is compatible with the claimed formula in (ii) in this case. Finally, if we additionally have that $|u_0|>1-|v_0|$, we apply Corollary \ref{6.3} (iii), and Corollary \ref{6.3} (ii) for $s=1$, to establish the asymptotic behaviors of $\mathcal C_{r,s}$ and $\mathcal C_{1,1}$ respectively, and again find that $\text{qdim}[M^\varepsilon_{r,s}]=0$ under the hypotheses given, as claimed.   

To prove (iv), we use Corollary \ref{6.3} (iv) to obtain
\begin{equation}\label{quotient}
\frac{\mathcal{C}_{r,s}(\varepsilon; i t)}{   \mathcal{C}_{1,1}(\varepsilon; i t)}\sim 
\frac{ \sum_\pm \pm\cos \left( \frac{\pi}{2p} \left( s+pr \mp\frac{u_0}{2pt} \right) \right) e^{\pm\frac{\pi i (2k+1) \left(s+pr \right)}{2p}}}{
	\sum_{\pm} \pm\cos \left( \frac{\pi}{2p} \left( 1+p \mp\frac{u_0}{2pt} \right) \right) e^{\pm\frac{\pi i (2k+1) \left(1+p \right)}{2p}}}.
\end{equation}
With $\alpha:=e^{\frac{\pi i (2k+1) \left(s+pr \right)}{2p}}$, $\beta:=e^{\frac{\pi i (2k+1) \left(1+p \right)}{2p}}$,
$a:=\pi \frac{s+pr}{2p}$, $b:=\pi \frac{1+p}{2p}$, and $T:= \frac{-u_0 \pi}{4p^2 t}$ (so that $T \rightarrow \infty$), \eqref{quotient} equals
\begin{equation*}
\frac{\sum_{\pm}\pm\cos(a\pm T)\alpha^{\pm 1}}{\sum_{\pm}\pm\cos (b\pm T)\beta^{\pm 1}}= \frac{i\cos (a)\sin(\Arg(\alpha))-\sin(a)\cos(\Arg(\alpha))\tan(T)}
{i\cos (b)\sin(\Arg(\beta))-\sin(b)\cos(\Arg(\beta))\tan(T)}.
\end{equation*}

Thus, we need to investigate for which $A, B, C, D \in \mathbb{C}$ 
$$
\lim_{T \to \infty } \frac{A+B\tan(T)}{C+D\tan(T)}
$$
exists.
If we consider the special sequence $T=\pi(j+1/2)$ with $j \in \mathbb{Z}$, we see that the above limit must be $B/D$. At the same time with 
$T= \pi j$ with $ j \in \mathbb{Z}$, this limit equals $A/C$ if $C\neq 0$. Therefore we must have $AD=BC$.
In this case, the above quotient (and the limit) exists and equals $B/D$. Note that if $C=0$ and the limit exists, then $A=0$ { and we again obtain $B/D$}. In our situation, this gives the condition
\begin{align*}
\tan(\Arg(\alpha))\tan(b) = \tan(a)\tan(\Arg(\beta)).
\end{align*}
Substituting into this expression the definitions of $\alpha,\beta,a$ and $b$   gives the condition stated in Theorem \ref{qdim}.
Substituting the appropriate values for $B/D$ and simplifying gives the claimed limit. 

For ${\rm qdim}[F^\varepsilon_\lambda] $,  we first recall \eqref{Fl}
and hence 
$$\eta(\tau) \ch[F_{\lambda}^\varepsilon] (\tau) = e^{2\pi \varepsilon \left(\lambda-\frac{\alpha_0}{2}\right)}q^{\frac{1}{2} \left(\lambda-\frac{\alpha_0}{2}\right)^2}.$$ 
Therefore, $\eta(it ) \ch[F_{\lambda}^\varepsilon] (it ) \sim  e^{2\pi \varepsilon (\lambda-\frac{\alpha_0}{2})} + O(t)$, for all $\varepsilon$. The results claimed in Theorem \ref{qdim} pertaining to $\text{qdim}[F_\lambda^\varepsilon]$ follow again from Corollary \ref{6.3}.
\end{proof}
 
 \section{Future work}

This work has several possible extensions. Here we briefly propose its ``higher rank" generalization in connection with representation 
theory. 

As explained in \cite{CM}, the atypical singlet characters ${\rm ch}[M_{r,s}](\tau)$ are in fact parametrized by the elements of the dual lattice of 
 $L=\sqrt{2p} \mathbb{Z}$, which can be viewed as a dilation of the $\frak{sl}_2$ root lattice. 
This construction generalizes to higher rank simple Lie algebras. More precisely, 
for every root lattice $Q$ of ADE type and $p \in \mathbb{N}_{\geq 2}$, there exists a vertex operator algebras 
whose (atypical) irreducible characters are in one-to-one correspondence with the elements of the dual lattice of $\sqrt{2p}Q$ (for details see \cite{M}).
These characters are further studied in \cite{BM2}, where we denoted them by ${\rm ch}[W^0(p,\lambda)_Q]$.
 In parallel with \cite{CM}, it is straightforward to regularize them by inclusion of an additional (Jacobi) variable $\varepsilon \in \mathbb{C}^n$ \cite{CM2}. 
The resulting expression  ${\rm ch}[W^0(p,\lambda)_Q](\varepsilon;\tau)$, modulo a power of the Dedekind $\eta$-function, can be expressed in terms of certain higher rank (Jacobi) partial theta functions and their derivatives.

We propose to study asymptotic properties (as $t \rightarrow 0^+$) of ${\rm ch}[W^0(p,\lambda)_Q](\varepsilon; \tau)$. This requires a suitable extension of 
Theorem \ref{Fasex} to higher ranks, which we (jointly with T. Creutzig) intend to address and solve in our future work   (see also \cite{CM2}).

\end{document}